\documentclass[11pt]{article}
\usepackage{amssymb}
\usepackage{mathrsfs}
\usepackage{amsfonts}
\usepackage{amsmath}
\numberwithin{equation}{section}
\allowdisplaybreaks[4]
\newtheorem{Theorem}{Theorem}[section]
\newtheorem{Proposition}[Theorem]{Proposition}
\newtheorem{Corollary}[Theorem]{Corollary}
\newtheorem{Lemma}[Theorem]{Lemma}
\newtheorem{Example}[Theorem]{Example}
\newtheorem{Remark}[Theorem]{Remark}
\newtheorem{Definition}[Theorem]{Definition}

\unless\ifcsname proof\endcsname

\fi

\def\bbZ{\mathbb{Z}}
\def\bbR{\mathbb{R}}
\def\bbT{\mathbb{T}}
\def \essinf{\mathbb{\mathrm{essinf}}}
\def \esssup{\mathbb{\mathrm{esssup}}}

\begin{document}

\title{Multilinear Fourier multipliers on variable Lebesgue spaces
\thanks{This work was supported partially by the
National Natural Science Foundation of China(10990012) and the Research Fund for the Doctoral Program
of Higher Education.}}

\author{Jineng Ren \quad and \quad Wenchang Sun\thanks{Corresponding author.}\\
School of Mathematical Sciences and LPMC,  Nankai University,
      Tianjin~300071, China\\
Email:\, renjineng@mail.nankai.edu.cn, sunwch@nankai.edu.cn}
\date{}
\maketitle

\begin{abstract} 
In this  paper, we study properties of the bilinear multiplier space. 
We give a necessary condition for a continuous integrable function to be a bilinear multiplier on variable exponent Lebesgue spaces. 
And we prove the localization theorem of multipliers on variable exponent Lebesgue spaces. 
Moreover, we present a Mihlin-H\"ormander type theorem for multilinear Fourier multipliers on weighted variable Lebesgue spaces and give 
some applications.
\end{abstract}



 \section{Introduction}

 Given a non-empty open set $\Omega \in\bbR^n$, we denote by
$\mathcal{P}(\Omega)$ the set of the variable exponent functions $p(x)$ such that
\begin{displaymath}
      1\leq p_-\leq p_+<\infty,
\end{displaymath}
where $p_-(\Omega):= \essinf\{p(x):x \in \Omega\}$ and $p_+(\Omega):= \esssup\{p(x):x \in \Omega\}$.

Let $\mathcal{P}^0(\Omega)$ be the set of exponent functions $p(x)$ such that
\begin{displaymath}
       0<p_-\leq p_+<\infty.
\end{displaymath}

Given a measurable functions $f$ on $\Omega$, for $1\leq p(\cdot)\leq \infty$, we define the modular functional associated with $p(\cdot)$ by
 \begin{displaymath}
 \rho_{p(\cdot),\Omega}(f)=\int_{\Omega\setminus \Omega_{\infty}}|f(x)|^{p(x)}dx+\|f(x)\|_{L^{\infty}(\Omega_{\infty})},
 \end{displaymath}
 where $\Omega_{\infty}$ denotes the set of points in $\Omega$ on which $p(x)=\infty$.

 The variable exponent Lebesgue space $ L^{p(\cdot)}(\Omega)$ is defined to be the set of Lebesgue measurable functions $f$ on $\Omega$ satisfying $\rho_{p(\cdot),\Omega}(f/\lambda)<\infty$ for some $\lambda>0$.
 The norm of $f$ in the space is defined by
\begin{displaymath}
\|f\|_{L^{p(\cdot)}}=\inf{\{\lambda >0:  \rho_{p(\cdot),\Omega}(f/ \lambda)\leq 1\}}.
\end{displaymath}
In the case that $p(\cdot)\in \mathcal{P}^0(\Omega)$, it is defined to be the set of all functions $f$ satisfying $|f(x)|^{p_0}\in L^{q(\cdot)}(\Omega)$, $q(x)=p(x)/p_0\in \mathcal{P}(\Omega)$ for some $0<p_0<p_-$(see \cite{MR2810550}). A quasi-norm in the space is defined by
\begin{displaymath}
\|f\|_{p(\cdot),\Omega}=\||f|^{p_0}\|_{q(\cdot),\Omega}^{\frac{1}{p_0}}.
\end{displaymath}
We refer to \cite{cruz} for an introduction to variable exponent Lebesgue spaces.

  In this paper, we study some properties of the space of bilinear Fourier multipliers and the Mihlin-H\"ormander type theorem for multilinear Fourier multipliers on weighted variable Lebesgue spaces. Specifically, let $m$ satisfy certain conditions. We discuss the N-linear Fourier multiplier operator $\mathrm{T}_m$ defined by

$
\mathrm{T}_m(f_1,\cdots,f_N)(x)$
\begin{displaymath}
{}\quad \quad =\int_{\bbR^{Nn}}e^{2 \pi i \langle\xi_1+\cdots+\xi_N,x\rangle}m(\xi_1,\cdots, \xi_N)\hat{f_1}(\xi_1),\cdots,\hat{f_N}(\xi_N)d\xi_1\cdots d\xi_N,
\end{displaymath}
for $f_1,\cdots,f_N \in \mathcal{S}(\bbR^n).$

The multilinear Fourier multipliers have been studied for a long
time. In \cite{MR518170}, Coifman and Meyer proved that $\mathrm{T}_m$ is bounded form $L^{p_1}(\bbR^n)\times\cdots \times L^{p_N}(\bbR^n)$ to $L^{p}(\bbR^n)$ for all $1<p_1,\cdots,p_N, p<\infty$ with $\frac{1}{p_1}+\cdots+\frac{1}{p_N}=\frac{1}{p}$ and $m\in C^s(\bbR^{Nn}\setminus \{0\})$ satisfying
\begin{equation}\label{1.1}
|\partial_{\xi_1}^{\alpha_1}\cdots \partial_{\xi_N}^{\alpha_N}m(\xi_1,\cdots,\xi_N
)|\leq C_{\alpha_1,\cdots,\alpha_N}(|\xi_1|+\cdots+|\xi_N|)^{-(|\alpha_1|+\cdots+|\alpha_N|)},
\end{equation}
for all $|\alpha_1|+\cdots+|\alpha_N|\leq s$, where $N \geq 2$ is an integer and $s$ is a sufficiently large integer.

Tomita\cite{MR2671120} gave a H\"ormander type theorem for multilinear multipliers. Specifically, $\mathrm{T}_m$ is bounded from $L^{p_1}(\bbR^n)\times\cdots \times L^{p_N}(\bbR^n)$ to $L^{p}(\bbR^n)$ for all $1<p_1,\cdots,p_N, p<\infty$ with $\frac{1}{p_1}+\cdots+\frac{1}{p_N}=\frac{1}{p}$ and $s=\frac{Nn}{2}+1$ in (\ref{1.1}). And Grafakos and Si studied the case $p \leq 1$ in \cite{Loukas&si}. The boundedness of multilinear Calder\'on-Zygmund operators with multiple weights was achieved by Grafakos, and Liu and Maldonado and Yang\cite{MR3183648}.

Under the H\"ormander conditions, Fujita and Tomita\cite{MR2958938} obtained some weighted estimates of $\mathrm{T}_m$ for classical $A_p$ weights. And in \cite{MR3178381}, Chen and Lu proved H\"ormander type multilinear theorem on weighted Lebesgue spaces when the Fourier multipliers was only assumed with limited smoothness.

In \cite{MR3007100}, the boundedness of $\mathrm{T}_m$ with multiple weights satisfying condition (\ref{1.1}) was given by Bui and Duong.

In \cite{MR3161113}, Li and Sun got some weighted estimates of $\mathrm{T}_m$ with multiple weights under the H\"ormander conditions in terms of the Sobolev regularity.

Huang and Xu\cite{MR2606534} obtained the boundedness of multilinear Calder\'on-Zygmund aperators on variable exponent Lebesgue spaces.

In this paper, we study the weighted estimates of $\mathrm{T}_m$ with nearly the same conditions as in \cite{MR3161113}, but on variable exponent Lebesgue spaces.

 The theory of bilinear multipliers was first studied by Coifman and Meyer \cite{MR576041}. They considered the ones with smooth symbols. Then, Tao and Thiele achieved some new results for non-smooth symbols in \cite{MR1887641}.

The study of bilinear multipliers has experienced a big progress since Lacey and Thiele\cite{MR1491450,MR1689336} proved that  $m(\xi,\nu)=sign(\xi+\alpha \nu)$ are $(p_1,p_2)$-multipliers for each couple $(p_1,p_2)$ such that $1<p_1,p_2 \leq \infty$, $p_3> 2/3$ and each $\alpha \in \bbR \setminus {\{0,1\}}$.

In \cite{MR1808390}, Fan and Sato proved the DeLeeuw type theorems for the transference of multilinear operators on Lebesgue and Hardy spaces from $\bbR^n$ to $\bbT^n$. In \cite{MR2172393}, Blasco gave the transference theorems from $\bbR^n$ to $\bbZ^n$. We also refer to \cite{R&T&Z,notes} for details.

We first give some definitions.

\begin{Definition}[\cite{weightandvariable}]
Let $1 \leq p_1(\cdot), p_2(\cdot) \leq \infty $, $0< p_3(\cdot) \leq \infty$, and $m(\xi, \eta)$ be a locally integrable function on $\bbR^{2n}$. Define
\[\mathrm{B}_m(f,g)(x)= \int_ {\bbR^{n}}\int_ {\bbR^{n}}\hat{f}(\xi)\hat{g}(\eta)m(\xi, \eta)e^{2 \pi i \langle\xi+\eta,x\rangle}d\xi d\eta\]
\textrm{for all} $f$ and $g$ such that $\hat{f}$ and $\hat{g}$ are bounded and compactly supported.
\end{Definition}

We call $m$ a bilinear multiplier on $\bbR^{2n}$ of type $(p_1(\cdot), p_2(\cdot), p_3(\cdot))$  if there exist some $C>0$ such that $\|\mathrm{B}_m(f,g)\|_{p_3(\cdot)}$$\leq$$C\|f\|_{p_1(\cdot)}\|g\|_{p_2(\cdot)}$ \textrm{for all} $f$ and $g$ such that $\hat{f}$, $\hat{g}$ are bounded and compactly supported, i.e. $\mathrm{B}_m$ extends to a bounded bilinear operator from $L^{p_1(\cdot)}$$\times$$L^{p_2(\cdot)}$ to $L^{p_3(\cdot)}$.

 We write $BM(\bbR^{2n})(p_1(\cdot),p_2(\cdot),p_3(\cdot))$ for the space of bilinear multipliers of type $(p_1(\cdot),p_2(\cdot),p_3(\cdot))$. Let $\|m\|_{p_1(\cdot),p_2(\cdot), p_3(\cdot)}=\|\mathrm{B}_m\|$.

 A similar function space is defined in the following.

\begin{Definition}
Given a function M on $\bbR^{n}$, let $m(\xi,\eta)$=$M(\xi-\eta)$. We say that
\begin{displaymath}
M\in \tilde{BM}(\bbR^{2n})(p_1(\cdot),p_2(\cdot),p_3(\cdot)), \end{displaymath}
if $\mathrm{B}_M(f,g)(x)=\int_{\bbR^{2n}}$ $\hat{f}(\xi)$$\hat{g}(\eta)$$M(\xi-\eta)e^{2\pi i\langle\xi+\eta,x\rangle}d\xi d\eta$ \textrm{for all} $f$ and $g$ such that $\hat{f}$, $\hat{g}$ are bounded and compactly supported, can be extended to a bounded linear operator form $L^{p_1(\cdot)}$$\times$
$L^{p_2(\cdot)}$ to $ L^{p_3(\cdot)}$.
\end{Definition}

\begin{Definition}
A function $p:\Omega \to \bbR^1$ is said to belong to the class $LH_0(\Omega)$, if
\[
|p(x)-p(y)|\leq \frac{C}{-\ln(|x-y|)}, \quad |x-y|\leq \frac{1}{2}, \quad x, y\in \Omega,
\]
where $C>0$ is independent of $x$ or $y$.
\end{Definition}

We simply write $LH_0$ instead of $LH_0(\bbR^n)$ if there is no confusion.
By $C$ etc., we denote various positive constants, which may have different values even in the same line.

\section{Some results on the space $BM(\bbR^{2n})(p_1(\cdot),p_2(\cdot),p_3(\cdot))$}
Some properties of the bilinear multiplier space on variable spaces were given by Kulak\cite{weightandvariable}. Here we give some other properties.

 First, we introduce the standard kernel.

\begin{Definition}
Given a function $K \in L_{loc}^1(\bbR^n\setminus \{0\})$, it is called a standard kernel, if there exists a constant $C>0$ such that:

\begin{enumerate}
\item $|K(x)|\leq \frac{C}{|x|^n}, \quad x\ne 0;$

\item $ |\nabla K(x)|\leq \frac{C}{|x|^{n+1}},\quad x\ne 0;$

\item for $0<r<R$,$\quad |\int_{\{r<|x|<R\}}K(x)dx|\leq C;$

\item $\lim_{\varepsilon \to 0}\int_{\{\varepsilon<|x|<1\}}K(x)dx$ exists.

\end{enumerate}
\end{Definition}

 \begin{Theorem}[localization]\label{local}
 Suppose that
\[m\in BM(\bbR^{2n})(p_1(\cdot),p_2(\cdot),p_3(\cdot)),\]
 $Q$ is a rectangle in $\bbR^{2n}$ and that the Hardy-Littlewood maximal operator $\mathcal{M}$ is bounded on $ L^{p_i(\cdot)}$, where $1<(p_i)_{-}\leq (p_i)_{+}<\infty$, $i=1,2$. Then
 \begin{displaymath}
 m\chi_Q \in BM(\bbR^{2n})(p_1(\cdot),p_2(\cdot),p_3(\cdot))
 \end{displaymath}
 and $\| m\chi_Q \|_{p_1(\cdot),p_2(\cdot),p_3(\cdot)}\leq C \| m \|_{p_1(\cdot),p_2(\cdot),p_3(\cdot)}$, where $C$ is independent of $Q$.
\end{Theorem}

 Let $BM(\bbR^n)(p(\cdot),p(\cdot))$ denote the space of multipliers which correspond to bounded operators from $L^{p(\cdot)}$ to $L^{p(\cdot)}$.

 To prove Theorem~\ref{local} we need the following results in the theory of variable Lebesgue spaces.

 \begin{Lemma}[{\cite[Theorem 5.39]{cruz}}]\label{Th1}
 Let $T$ be a singular integral operator with a standard kernel $K$. Given $p(\cdot) \in \mathcal{P}(\bbR^{n})$ such that $1<p_{-}\leq p_{+}<\infty$, if $\mathcal{M}$ is bounded on $L^{p(\cdot)}$, then for all functions f that are bounded and have compact support, $\| Tf \|_{p(\cdot)}$ $\leq$ $C\| f \|_{p(\cdot)}$, and $T$ extends to a bounded operator on $L^{p(\cdot)}$.
 \end{Lemma}

 \begin{Theorem}\label{Th2}
Suppose that
$ m_1 \in BM(\bbR^{n})(s_1(\cdot),p_1(\cdot))$, $m_2 \in BM(\bbR^{n})(s_2(\cdot),p_2(\cdot))$ and $
m \in
 BM(\bbR^{2n})(p_1(\cdot),p_2(\cdot),p_3(\cdot))$.
Then we have 
\[
 m_1(\xi)m(\xi,\eta)m_2(\eta)\in
 BM(\bbR^{2n})(s_1(\cdot),s_2(\cdot),p_3(\cdot)). 
\]
\end{Theorem}

 \begin{proof}
 For any $f$ and $g$ with $\hat{f},\hat{g} \in$ $C_c^{\infty}(\bbR^{n})$, we have
 \begin{flalign*}
\mathrm{B}_{m_1 m m_2}(f,g)(x)
&=\int_ {\bbR^{n}}\int_ {\bbR^{n}}\hat{f}(\xi)\hat{g}(\eta)m_1(\xi) m(\xi, \eta) m_2(\eta)e^{2 \pi i \langle\xi+\eta,x\rangle}d\xi d\eta\\
&=\int_ {\bbR^{n}}\int_ {\bbR^{n}} ((\mathrm{T}_{m_1}f))^{\land}(\xi)((\mathrm{T}_{m_2}g))^{\land}(\eta)m(\xi, \eta)e^{2 \pi i \langle\xi+\eta,x\rangle}d\xi d\eta\\
&=\mathrm{B}_{m}(\mathrm{T}_{m_1}f,\mathrm{T}_{m_2}g)(x).
\end{flalign*}

 Therefore,
\begin{eqnarray*}
 \|\mathrm{B}_{m_1 m m_2}(f,g)\|_{p_3(\cdot)}&\leq&\|\mathrm{B}_m\|\|\mathrm{T}_{m_1}f\|_{p_1(\cdot)}
 \|\mathrm{T}_{m_2}g\|_{p_2(\cdot)}\\
&\leq&\|\mathrm{B}_m\|\|m_1\|_{s_1(\cdot),p_1(\cdot)}
 \|m_2\|_{s_2(\cdot),p_2(\cdot)}\|f\|_{s_1(\cdot)}\|g\|_{s_2(\cdot)}.
 \end{eqnarray*}
Then we get the result.
\end{proof}

The following is an explicit example.

\begin{Example}
Suppose that $\frac{1}{p_1(\cdot)}$$+$$\frac{1}{p_2(\cdot)}$=$\frac{1}{p_3(\cdot)}$, $m_1 \in BM(\bbR^{n})(p_1(\cdot),p_1(\cdot))$ and $m_2 \in BM(\bbR^{n})(p_2(\cdot),p_2(\cdot))$, where $p_1(\cdot), p_2(\cdot)\in \mathcal{P}(\bbR^n)$ and $p_3(\cdot)\in \mathcal{P}^0(\bbR^n)$. Then
\begin{displaymath}
m(\xi,\eta)=m_1(\xi)m_2(\eta)\in BM(\bbR^{2n})(p_1(\cdot),p_2(\cdot),p_3(\cdot)).
\end{displaymath}
\end{Example}

\begin{proof}
For any $f$ and $g$ with $\hat{f},\hat{g} \in C_c^{\infty}(\bbR^{n})$, we have
\begin{eqnarray*}
 \mathrm{B}_{1}(f,g)(x)&=&\int_ {\bbR^{n}}\int_ {\bbR^{n}}\hat{f}(\xi)\hat{g}(\eta)e^{2 \pi i \langle\xi+\eta,x\rangle}d\xi d\eta\\
 &=&\int_ {\bbR^{n}}\int_
 {\bbR^{n}}\hat{f}(\xi)e^{2 \pi i \langle\xi,x\rangle}\hat{g}(\eta)e^{2 \pi i \langle\eta,x\rangle}d\xi d\eta\\
 &=&f(x)g(x).
\end{eqnarray*}
By  H\"older's inequality(\cite{cruz}), we have
\begin{eqnarray*} \|\mathrm{B}_{1}(f,g)(x)\|_{p_3(\cdot)}&=&\|f(x)g(x)\|_{p_3(\cdot)}\\
 &\leq & C\|f\|_{p_1(\cdot)}\|g\|_{p_2(\cdot)}.
\end{eqnarray*}
Thus 1$\in  BM(\bbR^{2n})(p_1(\cdot),p_2(\cdot),p_3(\cdot))$. By Theorem~\ref{Th2}, we have
 \begin{displaymath}
 \quad m(\xi,\eta)=m_1(\xi)m_2(\eta)\in BM(\bbR^{2n})(p_1(\cdot),p_2(\cdot),p_3(\cdot)).
\end{displaymath}
\end{proof}

 \begin{proof}[Proof of Theorem~\ref{local}]
 We only consider the case $n=1$. Other cases can be proved similarly. Suppose that $Q=[a,b]\times [c,d]$. Then for any $f$ and $g$ with $\hat{f},\hat{g} \in$ $C_c^{\infty}(\bbR^{n})$,
\begin{align*}
\mathrm{B}_{m\chi_{Q}}(f,g)(x)\\
&= \int_ {\bbR^{n}}\int_ {\bbR^{n}}\hat{f}(\xi)\hat{g}(\eta) m(\xi, \eta)\chi_{Q}(\xi, \eta)e^{2 \pi i \langle\xi+\eta,x\rangle}d\xi d\eta\\
 &=\int_ {\bbR^{n}}\int_ {\bbR^{n}}\hat{f}(\xi)\chi_{[a,b]}(\xi)\hat{g}(\eta)\chi_{[c,d]}(\eta)
   m(\xi, \eta)e^{2 \pi i \langle\xi+\eta,x\rangle}d\xi d\eta\\
&=\mathrm{B}_{m}((\hat{f}\chi_{[a,b]})^{\vee},(\hat{g}\chi_{[c,d]})^{\vee})(x).
 \end{align*}

Note that$(\hat{f}\chi_{[a,b]})^{\vee}=\frac{i}{2}(\mathrm{M}^a\mathrm{H}\mathrm{M}^{-a}-\mathrm{M}^b\mathrm{H}\mathrm{M}^{-b})f$, where $\mathrm{M}^{a}$ denotes the operator $\mathrm{M}^af(x)=e^{2 \pi i ax}f(x)$ and $\mathrm{H}$ denotes the Hilbert transform operator.
Since the Hilbert transform has a standard kernel, by Lemma \ref{Th1} we have
\begin{eqnarray*}
\|(\hat{f}\chi_{[a,b]})^{\vee}\|_{p_1(\cdot)}&=&\frac{1}{2}\|(\mathrm{M}^a\mathrm{H}\mathrm{M}^{-a}f
-\mathrm{M}^b\mathrm{H}\mathrm{M}^{-b}f)\|_{p_1(\cdot)}\\
&\leq& \frac{1}{2}\|\mathrm{H}\mathrm{M}^{-a}f\|_{p_1(\cdot)}+\frac{1}{2}
\|\mathrm{H}\mathrm{M}^{-b}f\|_{p_1(\cdot)}\\
&\leq & C\|f\|_{p_1(\cdot)}.
\end{eqnarray*}
So
  \[\chi_{[a,b]}\in BM(\bbR^{n})(p_1(\cdot),p_1(\cdot)).\]
Similarly we can prove that
\[\chi_{[c,d]}\in BM(\bbR^{n})(p_2(\cdot),p_2(\cdot)).\]
Hence by Theorem~\ref{Th2}, we get
\begin{displaymath}
m\chi_{Q} \in BM(\bbR^{2n})(p_1(\cdot),p_2(\cdot),p_3(\cdot)),
\end{displaymath}
and $\|m\chi_{Q}\|_{p_1(\cdot),p_2(\cdot),p_3(\cdot)}\leq C\|m\|_{p_1(\cdot),p_2(\cdot),p_3(\cdot)}$.

 \end{proof}

Next we show that the space $\tilde{BM}(\bbR^{2n})(p_1(\cdot),p_2(\cdot),p_3(\cdot))$ is invariant under certain operators.

\begin{Theorem}
Given $p_3(\cdot)\in \mathcal{P}(\bbR^n)$, $\phi \in L^1(\bbR^{n})$, if
\[
M\in\tilde{BM}(\bbR^{2n})(p_1(\cdot),p_2(\cdot),p_3(\cdot)),
\]
then \begin{displaymath}
\phi \ast M \in \tilde{BM}(\bbR^{2n})(p_1(\cdot),p_2(\cdot),p_3(\cdot)),
 \end{displaymath}
and $\| \phi \ast M\|_{p_1(\cdot),p_2(\cdot),p_3(\cdot)}$$\leq$$C\| \phi\|_{1}$$\| M\|_{p_1(\cdot),p_2(\cdot),p_3(\cdot)}$.
\end{Theorem}

\begin{proof}
For any $f$ and $g$ with $\hat{f},\hat{g} \in C_c^{\infty}(\bbR^{n})$, we have
\begin{align*}
  \mathrm{B}_{\phi \ast M}(f,g)(x)\\
 =&\int_ {\bbR^{2n}}\hat{f}(\xi)\hat{g}(\eta)\Big(\int_{\bbR^{n}}M(\xi-\eta-u)\phi(u)du\Big)e^{2 \pi i \langle\xi+\eta,x\rangle}d\xi d\eta\\
 =&\int_{\bbR^{n}}\int_{\bbR^{2n}}\widehat{\mathrm{M}^{-u}f}(\xi) \hat{g}(\eta)M(\xi-\eta)e^{2 \pi i \langle\xi+\eta,x\rangle}d\xi d\eta e^{2 \pi i \langle u, x\rangle}\phi(u) du.
\end{align*}
 By Minkowski's inequality,
\begin{eqnarray*}
\|\mathrm{B}_{\phi \ast M}(f,g)(x)\|_{p_3(\cdot)}&\leq&C \int_{\bbR^{n}}\|\mathrm{B}_{M}(\mathrm{M}^{-u}f,g)(x)\|_{p_3(\cdot)}|\phi(u)|du\\
&\leq&C\|M\|_{p_1(\cdot),p_2(\cdot),p_3(\cdot)}\|\phi\|_1\|f\|_{p_1(\cdot)}
\|g\|_{p_2(\cdot)}.
\end{eqnarray*}
\end{proof}

\begin{Theorem}
Suppose that $p_3\geq 1$, $M$$\in$$\tilde{BM}$$(\bbR^{2n})$$(p_1(\cdot),p_2(\cdot),p_3)$ and $\phi \in L^1(\bbR^{n})$. Then
\begin{displaymath}
 m(\xi, \eta):=M(\xi-\eta)\hat{\phi}(\xi+\eta)\in BM(\bbR^{2n})(p_1(\cdot),p_2(\cdot),p_3),
\end{displaymath}
and $\| m \|_{p_1(\cdot),p_2(\cdot),p_3}$$\leq$$\|\phi\|_1 \|M \|_{p_1(\cdot),p_2(\cdot),p_3}$.
\end{Theorem}

\begin{proof}
For any $f$ and $g$ with $\hat{f},\hat{g} \in C_c^{\infty}(\bbR^{n})$, we have
\begin{align*}
 \mathrm{B}_{m}(f,g)(x)\\
=&\int_ {\bbR^{2n}}\hat{f}(\xi)\hat{g}(\eta)M(\xi-\eta)\Big(\int_{\bbR^{n}}\phi(y)e^{-2 \pi i \langle\xi+\eta,y\rangle}dy\Big)e^{2 \pi i \langle\xi+\eta,x\rangle}d\xi d\eta\\
=&\int_{\bbR^{n}} \Big(\int_ {\bbR^{2n}}\hat{f}(\xi)\hat{g}(\eta)M(\xi-\eta)e^{2 \pi i \langle\xi +\eta, x-y\rangle}d\xi d\eta\Big)\phi(y)dy\\
=&\phi \ast \mathrm{B}_{M}(f,g)(x).
\end{align*}
By Young's inequality, we have
\begin{eqnarray*}
\|\mathrm{B}_{m}(f,g)\|_{p_3}&\leq&\|\phi\|_1\|\mathrm{B}_{M}(f,g)\|_{p_3}\\
&=&\|\phi\|_1\|\mathrm{B}_M\|\|f\|_{p_1(\cdot)}\|g\|_{p_2(\cdot)}.
\end{eqnarray*}
Thus, we get the conclusion.
\end{proof}

Finally, we consider the necessary condition of this kind of multipliers. The bilinear classical counterpart was obtained by H\"ormander\cite[Theorem 3.1]{MR0121655} and Blaso\cite{MR2656523}. The multilinear classical one was proved by Grafakos and Torres, see \cite[Proposition 5]{MR1880324} and \cite[Proposition 2.1]{MR2595656}. And the one for multipliers on Lorentz spaces was given by Villarroya\cite[Proposition 3.1]{MR2471164}. Some of their proofs used the translation-invariant property of the classical spaces, which is, however, no longer valid on $L^{p(\cdot)}$. In the following, we prove the variable version of the necessary condition.

\begin{Theorem}[necessary condition]\label{necessary}
Let $|\Omega_\infty^i|=0$, where $\Omega_\infty^i$ denote the set of points on which $p_i(x)=\infty$, $i=1,2,3$. Suppose that there is a non-zero continuous integrable function $M$ such that $M \in  \tilde{BM}(\bbR^{2n})$$(p_1(\cdot),p_2(\cdot),p_3(\cdot))$. Then
\begin{displaymath}
\frac{1}{(p_3)_{+}}\leq\frac{1}{(p_1)_{-}}+\frac{1}{(p_2)_{-}}.
\end{displaymath}
\end{Theorem}

To prove the theorem, we need the following results.

 \begin{Proposition}[{\cite[Corollary~2.22]{cruz}}]\label{Cruz1}
Fix $\Omega$ and $1\leq p(\cdot)\leq \infty$. If $\|f\|_{p(\cdot)}\leq 1$, then $\rho(f)\leq \|f\|_{p(\cdot)}$; if $\|f\|_{p(\cdot)}>1$, then $\rho(f)\geq\|f\|_{p(\cdot)}$.
\end{Proposition}

\begin{Proposition}[{\cite[Corollary~2.23]{cruz}}]\label{Cruz2}
Given $\Omega$ and $1\leq p(\cdot)\leq \infty$, suppose $|\Omega_\infty|=0$. If $\|f\|_{p(\cdot)}>1$, then
\[
\rho(f)^{1/{p_+}}\leq \|f\|_{p(\cdot)}\leq \rho(f)^{1/{p_-}}.
\]
If $0<\|f\|_{p(\cdot)}\leq1$, then
\[
\rho(f)^{1/{p_-}}\leq \|f\|_{p(\cdot)}\leq \rho(f)^{1/{p_+}}.
\]
\end{Proposition}

\begin{Lemma}\label{L2}
Let $M \in  \tilde{BM}(\bbR^{2n})(p_1(\cdot),p_2(\cdot),p_3(\cdot))$. If $\frac{1}{q}=\frac{1}{(p_1)_{-}}+\frac{1}{(p_2)_{-}}-\frac{1}{(p_3)_{+}}$ and $|\Omega_\infty^i|=0$, $i=1,2,3$, then there exist some $C>0$ such that
\[
\Big|\lambda^n\int_{\bbR^{n}}e^{-\lambda^2\xi^2}M(\xi)d\xi\Big|\leq C\|M\|_{p_1(\cdot),p_2(\cdot),p_3(\cdot)}\lambda^{\frac{n}{q}},
\]
when $\lambda$ is sufficiently large.
\end{Lemma}

 \begin{proof}
 Let $\lambda>0$. Define $G_{\lambda}$ by $\hat{G_{\lambda}}(\xi)=e^{-2 \lambda^2\xi^2}$. By a simple change of variable, one gets that
\begin{eqnarray} \label{eq:1}
\nonumber \mathrm{B}_{M}(G_\lambda,G_\lambda)(x)&=&\frac{1}{2}\int_{\bbR^{2n}}e^{ -\lambda^2\nu^2}e^{ -\lambda^2\mu^2}M(\nu)e^{2 \pi i\mu x}d\mu d\nu\\ &=&\frac{C}{\lambda^n}e^{-\pi^2|\frac{x}{\lambda}|^2}\int_{\bbR^{n}}e^{-\lambda^2\nu^2}M(\nu)d\nu,
\end{eqnarray}
where we use the fact that $G_{\lambda}(x)=(e^{-2
  \lambda^2\xi^2})^{\vee}=\frac{C}{\lambda^n}e^{-\frac{\pi^2}{2}|\frac{x}{\lambda}|^2}$.

  Observe that
  \begin{eqnarray*}
 \rho_{p_i(\cdot)}(e^{-\frac{\pi^2}{2}|\frac{x}{\lambda}|^2})&=&\int_{\bbR^{n}}e^{-\frac{\pi^2}{2}|\frac{x}{\lambda}|^2p_i(x)}dx\\
  &=&\lambda^{n}\int_{\bbR^{n}}e^{-\frac{\pi^2}{2}|u|^2p_i(\lambda u)}du\\
  &\leq&\lambda^{n}\int_{\bbR^{n}}e^{-\frac{\pi^2}{2}|u|^2(p_i)_{-}}du\\
  &=&C_{(p_i)_{-}}\lambda^n,
  \end{eqnarray*}
where $i=1,2$.

Similarly we have
  \begin{displaymath}
  \quad \quad \rho_{p_i(\cdot)}(e^{-\frac{\pi^2}{2}|\frac{x}{\lambda}|^2})\geq C_{(p_i)_{+}}\lambda^n,\quad i=1,2.
  \end{displaymath}
By Proposition~\ref{Cruz1}, we get $\|e^{-\frac{\pi^2}{2}|\frac{x}{\lambda}|^2}\|_{p_i(\cdot)}>1$, when $\lambda$ is sufficiently large.
Thus by Proposition~\ref{Cruz2}, we have
\[
{\rho_{p_i(\cdot)}(e^{-\frac{\pi^2}{2}|\frac{x}{\lambda}|^2})}^{\frac{1}{(p_i)_{+}}}\leq\|e^{-\frac{\pi^2}{2}|\frac{x}{\lambda}|^2}\|_{p_i(\cdot)}\leq{\rho_{p_i(\cdot)}(e^{-\frac{\pi^2}{2}|\frac{x}{\lambda}|^2})}^{\frac{1}{(p_i)_{-}}}.
\]
So
\begin{equation}\label{eq:2}
C_{(p_i)_+}\lambda^{n/(p_i)_{+}-n}\leq \|G_\lambda\|_{p_i(\cdot)}\leq C_{(p_i)_-}\lambda^{n/(p_i)_{-}-n},
\end{equation}
where $i=1,2$.

Similarly we can get
  \begin{equation}\label{eq:3}
  C_{(p_3)_+}\lambda^{n/(p_3)_{+}-n}\leq \|\frac{1}{\lambda^n}e^{-{\pi^2}|\frac{x}{\lambda}|^2}\|_{p_3(\cdot)}\leq C_{(p_3)_-}\lambda^{n/(p_3)_{-}-n}.
  \end{equation}
All the inequalities above are established when the $\lambda$ is sufficiently large.

 By the assumption, we have
 \begin{equation} \label{eq:4}
  \|\mathrm{B}_M(G_\lambda,G_\lambda)\|_{p_3(\cdot)}\leq \|M\|_{p_1(\cdot),p_2(\cdot),p_3(\cdot)}\|G_\lambda\|_{p_1(\cdot)}\|G_\lambda\|_{p_2(\cdot)}.
  \end{equation}

 Now combining (\ref{eq:1}), (\ref{eq:2}), (\ref{eq:3}) and (\ref{eq:4}), we get
\begin{eqnarray*}
C_{(p_3)_+}\lambda^{\frac{n}{(p_3)_{+}}-n}\Big|\int_{\bbR^{n}}e^{-\lambda^2\xi^2}M(\xi)d\xi\Big| &\leq & \|\mathrm{B}_M(G_\lambda,G_\lambda)\|_{p_3(\cdot)}\\
&\leq& C\|M\|_{p_1(\cdot),p_2(\cdot),p_3(\cdot)}\lambda^{\frac{n}{(p_1)_-}-n}\lambda^{\frac{n}{{(p_2)_-}}-n}.
\end{eqnarray*}

Hence
\[\Big|\lambda^n \int_{\bbR^{n}}e^{-\lambda^2\xi^2}M(\xi)d\xi\Big|\leq C\|M\|_{p_1(\cdot),p_2(\cdot),p_3(\cdot)}\lambda^{\frac{n}{q}},
\]
when $\lambda$ is sufficiently large.
 \end{proof}

 We are now ready to prove Theorem~\ref{necessary}.

 \begin{proof}[Proof of Theorem~\ref{necessary}]
 Assume that $\frac{1}{(p_1)_{-}}+\frac{1}{(p_2)_{-}}<\frac{1}{(p_3)_{+}}$.
 By a simple calculation, we obtain that
 \begin{displaymath}
 \mathrm{B}_M(\mathrm{M}^{y}f,\mathrm{M}^{-y}g)=\mathrm{B}_{\mathrm{T}_{-2y}M}(f,g),
 \end{displaymath}
 where $\mathrm{T}_{-2y}M=M(x+2y)$.
 Thus $\mathrm{T}_{-2y}M \in \tilde{BM}(\bbR^{2n})$$(p_1(\cdot),p_2(\cdot),p_3(\cdot))$. Applying Lemma~\ref{L2} to $\mathrm{T}_{-2y}M$, we get
\begin{displaymath}
\Big|\lambda^n\int_{\bbR^{n}}e^{-\lambda^2\xi^2}M(\xi+2y)d\xi\Big| \leq C\|M\|_{p_1(\cdot),p_2(\cdot),p_3(\cdot)}\lambda^{\frac{n}{q}}.
 \end{displaymath}
 Observe that $\frac{1}{q}=\frac{1}{(p_1)_{-}}+\frac{1}{(p_2)_{-}}-\frac{1}{(p_3)_{+}}<0$ and $M$ is continuous. By letting $\lambda \to \infty$, we have
 \begin{displaymath}
 \lim_{\lambda \rightarrow \infty}\Big|\lambda^n\int_{\bbR^{n}}e^{-\lambda^2\xi^2}M(\xi+2y)d\xi\Big|={\pi}^{\frac{n}{2}}|M(2y)|=0.
 \end{displaymath}
Since $y$ is arbitrary, we have $M=0$. This is a contradiction.
Thus
\begin{displaymath}
\frac{1}{(p_3)_{+}}\leq\frac{1}{(p_1)_{-}}+\frac{1}{(p_2)_{-}}.
\end{displaymath}
\end{proof}

\section{The Mihlin-H\"omander type estimate for multilinear multipliers on weighted variable exponent lebesgue spaces}
Roughly speaking, in the linear case, by adding the condition that the Hardy-Littlewood maximal function is bounded on weighted variable spaces, the results of multipliers on weighted variable spaces can be derived from the weighted multiplier theorem on classical Lebesgue spaces and the Extrapolation theorem on weighted variable spaces. See, for example, \cite [Theorem 4.5, Theorem 4.7]{MR2499882}, \cite{one-multiplier}, and \cite{one-multiplier1}.

However, in the multilinear case, the method faces some challenges. One problem is that we have no multilinear Extrapolation theorem on spaces with variable exponents yet, though the counterpart on classical Lebesgue spaces appeared early, see \cite{multi-extra}.

We give another way to get the Mihlin-H\"ormander conditions for multilinear Fourier multipliers on weighted variable spaces.

Recall that the Hardy-Littlewood maximal function is defined by
\begin{displaymath}
\mathcal{M}(f)(x)=\sup_{Q\ni x}\frac{1}{|Q|}\int_Q|f(y)|dy.
\end{displaymath}
And the sharp maximal function is defined by
\begin{displaymath}
\mathrm{M}^{\#}(f)(x)=\sup_{Q\ni x}\inf_{c\in \bbR}\frac{1}{|Q|}\int_Q|f(y)-c|dy.
\end{displaymath}
For $\delta>0$, we also define
\begin{displaymath}
\mathrm{M}_\delta(f)=\mathcal{M}(|f|^\delta)^{1/\delta}\quad and \quad \mathrm{M}_\delta^{\#}(f)=\mathrm{M}^{\#}(|f|^\delta)^{1/\delta}.
\end{displaymath}
For $\vec{f}=(f_1,\cdots,f_N)$ and $p \geq1$, we define
\[
\mathcal{M}_p(\vec{f})(x)=\sup_{Q\ni x}\prod_{i=1}^{N}\Big(\frac{1}{|Q|}\int_Q|f_i(y_i)|^pdy_i\Big)^{1/p}.
\]

\begin{Definition}[\cite{MR2483720}]
Given $\vec{P}=(p_1,\cdots,p_N)$ with $1\leq p_1,\cdots,p_N<\infty$ and $1/{p_1}+\cdots+1/{p_N}=1/p$. Let $\vec{w}=(w_1,\cdots,w_N)$. Set
\[
v_{\vec{w}}=\prod_{i=1}^{N}w_i^{p/{p_i}}.
\]
We say that $\vec{w}$ satisfies the $A_{\vec{P}}$ condition if
\[
\sup_Q\Big(\frac{1}{|Q|}\int_Q v_{\vec{w}}\Big)^{1/p} \prod_{i=1}^{N}\Big( \frac{1}{|Q|}\int_Q w_i^{1-p_i^\prime}\Big)^{1/{p_i^\prime}}<\infty.
\]
When $p_i=1$, then $\Big(\frac{1}{|Q|}\int_Q w_i^{1-p_i^\prime}\Big)^{1/{p_i^\prime}}$ is understood as $(\inf_Q w_i)^{-1}$.
\end{Definition}

We now give a Mihlin-H\"ormander type theorem for multilinear Fourier multipliers on weighted variable exponent Lebesgue spaces.

\begin{Theorem}\label{hormander}
Suppose that $Nn/2<s\leq Nn$, $m\in L^{\infty}(\bbR^{Nn})$ and
\begin{displaymath}
\sup_{R>0}\|m(R\xi)\chi_{\{1<|\xi|<2\}}\|_{H^s(\bbR^{Nn})}<\infty.
\end{displaymath}
Set $r_0:=Nn/s$, a series of variable indexes $p_1(x),\cdots,p_N(x)\in \mathcal{P}(\bbR^n)$, and $p(x)\in \mathcal{P}^0(\bbR^n)$, such that $\frac{1}{p_1(x)}+\frac{1}{p_2(x)}+\cdots+\frac{1}{p_N(x)}=\frac{1}{p(x)}$, where $(p_j)_->r_0$,$j=1,2,\cdots,N$. Suppose there are $0<q<p_-$, $r_0<q_j<(p_j)_-$ such that the Hardy-Littlewood maximal operator $\mathcal{M}$ is bounded on $L^{\tilde{p}^{\prime}(\cdot)}((w_1\cdots w_N)^{-q\tilde{p}^{\prime}(\cdot)})
$ and $L^{\tilde{p_j}^{\prime}(\cdot)}({w_j}^{-q_j\tilde{p_j}^{\prime}(\cdot)})$, where $\tilde{p}(x)=\frac{p(x)}{q}$, $\tilde{p_j}(x)=\frac{p_j(x)}{q_j}$, $j=1,2,\cdots,N$. Then there exist some $C>0$ such that
\begin{displaymath}
\|\mathrm{T}_m(\vec{f})\|_{L^{p(\cdot)}(w_1^{p(\cdot)}\cdots w_N^{p(\cdot)} )}\leq C \prod_{i=1}^{N}\|f_i\|_{L^{p_i(\cdot)}(w_i^{p_i(\cdot)})}.
\end{displaymath}
\end{Theorem}

Before proving the theorem, we present some preliminary results. The following inequality is a classical result of Fefferman and Stein \cite{MR0447953}.

\begin{Proposition}[\cite{MR0447953}]\label{pro1}
Let $0<\delta<p<\infty$ and $w \in A_\infty$. Then there exist some constants $C_{n,p,\delta,w}>0$ such that
\begin{displaymath}
\int_{\bbR^n}(\mathrm{M}_\delta f)(x)^pw(x)dx \leq C_{n,p,\delta,w}\int_{\bbR^n}(\mathrm{M}_\delta^{\#} f)(x)^pw(x)dx.
\end{displaymath}
\end{Proposition}

 The next result comes from  Lemma~2.6 in \cite{MR3161113}. For our purpose we restate it in the proper way.

\begin{Proposition}[\cite{MR3161113}]\label{pro2}
Let $1<r<min\{\frac{s}{(s-1)},\frac{2s}{Nn}\} $ such that $p_0:=rr_0<q_j$, $j=1,\cdots,N $. If $0<\delta<p_0/N$, then under the assumption of Theorem~\ref{hormander}, there exist some $C>0$ such that for all $\vec{f} \in L^{t_1}(\bbR^n)$$\times\cdots \times $$L^{t_N}(\bbR^n)$, $p_0 \leq t_1,\cdots, t_N<\infty$, we have
\begin{displaymath}
\mathrm{M}_\delta^{\#}(\mathrm{T}_m{\vec{f}})\leq C\mathcal{M}_{p_0}(\vec{f}).
\end{displaymath}
\end{Proposition}

\begin{Proposition}[\cite{MR2499882}]\label{pro3}
Let X be a metric measure space and $\Omega$ an open set in X. Assume that for some $p_0$ and $q_0$, satisfying:
\begin{displaymath}
0<p_0\leq q_0<\infty, p_0<p_-, and \frac{1}{p_0}-\frac{1}{p_+}<\frac{1}{q_0},
\end{displaymath}
and for every weight $w \in A_1(\Omega)$, there holds the inequality
\begin{displaymath}
\Big(\int_\Omega f^{q_0}(x)w(x)d\mu(x)\Big)^{\frac{1}{q_0}}\leq c_0\Big(\int_\Omega g^{p_0}(x)[w(x)]^{\frac{p_0}{q_0}}d\mu(x)\Big)^{\frac{1}{p_0}},
\end{displaymath}
for all $(f,g)$ in a given family $\mathscr{F}$.
Let the variable exponent $q(x)$ be defined by
\begin{displaymath}
\frac{1}{q(x)}=\frac{1}{p(x)}-(\frac{1}{p_0}-\frac{1}{q_0}).
\end{displaymath}
Let the exponent $p(x)$ and the weight $\varrho$ satisfy that $p\in \mathcal{P}^{0}(\Omega)$ and $\mathcal{M}$ is bounded on $L^{\tilde{q}^{\prime}(\cdot)}(\Omega,\varrho^{-q_0\tilde{q}^{\prime}(\cdot)})$.

Then for all $(f,g)\in \mathscr{F}$ with $f\in L^{q(\cdot)}(\Omega,\varrho^{q(\cdot)})$, the inequaltiy
\begin{displaymath}
\|f\|_{L^{q(\cdot)}(\Omega,\varrho^{q(\cdot)})}\leq\|g\|_{L^{p(\cdot)}(\Omega,\varrho^{p(\cdot)})}
\end{displaymath}
is valid with a constant $C>0$.
\end{Proposition}

\begin{Remark}
Note that the condition $p\in \mathcal{P}(\Omega)$ in the Extrapolation theorem of \cite{MR2499882} can be released to $p\in \mathcal{P}^0(\Omega)$ with nearly no modification to the proof.
\end{Remark}

\begin{Proposition}[{\cite[Proposition 2.3]{MR3007100}}]\label{pro4}
Let $p_0\geq1$ and $p_i>p_0$ for $i=1,\cdots,N$ and $1/{p_1}+\cdots+1/{p_N}=1/p$. Then the inequality
\[
\|\mathcal{M}_{p_0}\vec{(f)}\|_{L^{p}(v_{\vec{w}})}\leq C\prod_{i=1}^N\|f_i\|_{L^{p_i}(w_i)}
\]
holds if and only if $\vec{w}\in A_{\vec{P}/{p_0}}$, where $\vec{P}/{p_0}=(p_1/{p_0},\cdots,p_N/{p_0})$.
\end{Proposition}
 
\begin{Remark}
When $N=1$, the conclusion above is valid. Specifically, let $p_0\geq1$ and $p>p_0$, then $\|\mathcal{M}_{p_0}f\|_{L^{p}(w)}\leq C\|f\|_{L^{p}(w)}$ holds if and only if $w\in A_{p/{p_0}}$.
\end{Remark}

We are now ready to prove Theorem~\ref{hormander}

\begin{proof}[Proof of Theorem~\ref{hormander}]
For any $f_j$ with $\hat{f_j} \in C_c^\infty(\bbR^n)$, $j=1,\cdots,N $ and $v\in A_\infty$,
by Proposition~\ref{pro1} and Proposition~\ref{pro2}, we have
\begin{eqnarray} \label{first}
\nonumber\|\mathrm{T}_m(\vec{f})\|_{L^q(v)}&\leq &\|\mathrm{M}_\delta(\mathrm{T}_m(\vec{f}))\|_{L^q(v)}\\
\nonumber &\leq &C_{n,q,\delta,v}\|\mathrm{M}_\delta^{\#}(\mathrm{T}_m(\vec{f}))\|_{L^q(v)}\\
&\leq &C\|\mathcal{M}_{p_0}(\vec{f})\|_{L^q(v)},
\end{eqnarray}
where $p_0$ is defined as in Proposition~\ref{pro2}.

Since the maximal operator $\mathcal{M}$ is bounded on $L^{\tilde{p}^{\prime}(\cdot)}((w_1\cdots w_N)^{-q\tilde{p}^{\prime}(\cdot)})
$, by Proposition~\ref{pro3}, we have
\begin{eqnarray}\label{second}
\|\mathrm{T}_m(\vec{f})\|_{L^{p(\cdot)}(w_1^{p(\cdot)}\cdots w_N^{p(\cdot)} )}\leq C\|\mathcal{M}_{p_0}(\vec{f})\|_{L^{p(\cdot)}(w_1^{p(\cdot)}\cdots w_N^{p(\cdot)} )}.
\end{eqnarray}
By H\"older's inequality,
\begin{align}\label{holder}
\nonumber\|\mathcal{M}_{p_0}(\vec{f})\|_{L^{p(\cdot)}(w_1^{p(\cdot)}\cdots w_N^{p(\cdot)} )}=& \|\mathcal{M}_{p_0}(\vec{f})w_1\cdots w_N\|_{L^{p(\cdot)}}\\
\nonumber\leq&\|\prod_{i=1}^N\{\mathcal{M}_{p_0}(f_i)w_i\}\|_{L^{p(\cdot)}}\\
\leq&C\|\mathcal{M}_{p_0}(f_1)w_1\|_{L^{p_1(\cdot)}}\cdots\|\mathcal{M}_{p_0}(f_N)w_N\|_{L^{p_N(\cdot)}},
\end{align}
where
\begin{displaymath}
\mathcal{M}_{p_0}(f_i):=\sup_{Q\ni x}\Big(\frac{1}{|Q|}\int_Q|f_i(y_i)|^{p_0}dy_i\Big)^{\frac{1}{p_0}},
\quad i=1,\cdots,N.
\end{displaymath}
Since $p_0<q_j$, we can choose $u_j>1$ such that $p_0u_j=q_j$. Thus by Proposition~\ref{pro4}, we get
\begin{displaymath}
\|\mathcal{M}_{p_0}(f)\|_{L^{q_j}(w)}\leq C\|f\|_{L^{q_j}(w)}
\end{displaymath}
is valid for all $w\in A_{u_j}$, $f\in L^{q_j}(w)$. Using the boundedness of  $\mathcal{M}$ again, we see from Proposition~\ref{pro3} that
\begin{displaymath}
\|\mathcal{M}_{p_0}(f_j)\|_{L^{p_j(\cdot)}(w_j^{p_j(\cdot)})}\leq C\|f_j\|_{L^{p_j(\cdot)}(w_j^{p_j(\cdot)})}, j=1,\cdots,N.
\end{displaymath}
It follows from (\ref{holder}) that
\begin{displaymath}
\|\mathcal{M}_{p_0}(\vec{f})\|_{L^{p(\cdot)}(w_1^{p(\cdot)}\cdots w_N^{p(\cdot)} )}\leq C\|f_1\|_{L^{p_1(\cdot)}(w_1^{p_1(\cdot)})}\cdots\|f_N\|_{L^{p_N(\cdot)}(w_N^{p_N(\cdot)})}.
\end{displaymath}
By (\ref{second}), we obtain the desired conclusion,
\begin{displaymath}
\|\mathrm{T}_m(\vec{f})\|_{L^{p(\cdot)}(w_1^{p(\cdot)}\cdots w_N^{p(\cdot)} )}\leq C\|f_1\|_{L^{p_1(\cdot)}(w_1^{p_1(\cdot)})}\cdots\|f_N\|_{L^{p_N(\cdot)}(w_N^{p_N(\cdot)})}.
\end{displaymath}
\end{proof}

  As an application of Theorem~\ref{hormander}, we now consider the case when weight functions are defined by
\begin{equation}\label{weight}
w_j(x)= [1+|x-x_0|]^{\beta_\infty^j}\prod_{k=1}^l|x-x_k|^{\beta_k^j},  j=1,\cdots,N,
\end{equation}
where $x_k$ are fixed points in $\bbR^n$, $k=1,\cdots,l$.

\begin{Corollary}\label{last corollary}
Suppose that $Nn/2<s\leq Nn$, $m\in L^{\infty}(\bbR^{Nn})$ and
\begin{displaymath}
\sup_{R>0}\|m(R\xi)\chi_{\{1<|\xi|<2\}}\|_{H^s(\bbR^{Nn})}<\infty.
\end{displaymath}
Let the variable exponents $p_1(x),\cdots,p_N(x)$ and $p(x)$ satisfy that $\frac{1}{p_1(x)}+\frac{1}{p_2(x)}+\cdots+\frac{1}{p_N(x)}=\frac{1}{p(x)}$, where $1<p_-\leq p_+ <\infty$, $r_0:=Nn/s<(p_j)_-\leq (p_j)_+<\infty$, and $p_j \in LH_0(\bbR^n)$. Suppose that there exist some $R>0$ and $x_0\in \bbR^n$ such that $p_j(x)\equiv (p_j)_\infty=const$, for $x \in \bbR^n \setminus B(x_0,R)$, $j=1,\cdots,N $, and that
\[
 -\frac{n}{p_j(x_k)}<\beta_k^j<min\{ \frac{n}{p_j^\prime(x_k)}, \frac{n}{Np^\prime(x_k)}\},\quad k=1,\cdots,l,
\]
\[
 -\frac{n}{(p_j)_{\infty}}<\beta_\infty^j+\sum_{k=1}^l \beta_k^j<min\{\frac{n}{(p_j)_{\infty}^\prime}, \frac{n}{Np_\infty^\prime}\},
\]
for $j=1,\cdots,N$. Then $\mathrm{T}_m$ is bounded from $L^{p_1(\cdot)}(w_1^{p_1(\cdot)})\times\cdots \times L^{p_N(\cdot)}(w_N^{p_N(\cdot)})$ to $L^{p(\cdot)}(w_1^{p(\cdot)}\cdots w_N^{p(\cdot)})$.
\end{Corollary}

To prove Corollary~\ref{last corollary}, we need to define a class of weight functions, which is a special case of \cite[Definition~2.7]{MR2499882}.

\begin{Definition}[\cite{MR2499882}]
Let $p(\cdot)\in C(\bbR^n)$ and there exists $R>0$ and $x_0\in \bbR^n$, such that $p(x)\equiv p_\infty=const$, for all $x\in \bbR^n \setminus B(x_0,R)$. A weight function $w$ of the form
\[
w= [1+|x-x_0|]^{\beta_\infty}\prod_{k=1}^l|x-x_k|^{\beta_k}
\]
is said  to belong to the class $V_{p(\cdot)}(\bbR^n,\Pi)$, if
\[
-\frac{n}{p(x_k)}<\beta_k<\frac{n}{p^\prime(x_k)},\quad k=1,\cdots,l,
\]
 and
 \[
 -\frac{n}{p_\infty}<\beta_\infty +\sum_{k=1}^{l}\beta_k<\frac{n}{p_\infty^\prime}.
 \]
\end{Definition}

We have the following lemma.

\begin{Lemma}\label{last lemma}
Let $p(\cdot)\in LH_0$ satisfy $1<p_-\leq p_+<\infty$. Suppose that there exist some $R>0$ and $x_0\in \bbR^n$ such that $p(x)\equiv p_\infty=const$ for $x \in \bbR^n \setminus B(x_0,R)$. If $\varrho\in V_{p(\cdot)}(\bbR^n,\Pi)$, then $\mathcal{M}$ is bounded on the space $L^{(\tilde{p})^\prime(\cdot)}(\varrho^{-q_0(\tilde{p})^\prime (\cdot)})$ for all $q_0 \in (1,p_-)$, where $ \tilde{p}(\cdot)=\frac{p(\cdot)}{q_0}$.
\end{Lemma}

\begin{proof}
If $p(\cdot) \in LH_0$, then $\tilde{p}(\cdot)\in LH_0$. By \cite[Proposition~2.3]{cruz}, we have $(\tilde{p})^\prime(\cdot)\in LH_0$.
And since $\varrho\in V_{p(\cdot)}(\bbR^n,\Pi)$, by \cite[Remark~2.10]{MR2499882}, we know $\varrho^{-q_0} \in V_{(\tilde{p})^\prime(\cdot)}(\bbR^n,\Pi)$. Then it follows from \cite[Theorem~2.12]{MR2499882} that $\mathcal{M}$ is bounded on $L^{(\tilde{p})^\prime(\cdot)}(\varrho^{-q_0(\tilde{p})^\prime (\cdot)})$.
  \end{proof}

Now we are ready to prove Corollary~\ref{last corollary}.

\begin{proof}[Proof of Corollary~\ref{last corollary}]
Fix some $1<q<p_-$. Let $q_j$, $\tilde{p}(x)$ and $\tilde{p_j}(x)$ be defined as in Theorem~\ref{hormander}. By the assumption, we have
\[
-\frac{n}{p_j(x_k)}<\beta_k^j<\frac{n}{p_j^\prime(x_k)},\quad k=1,\cdots,l,
\]
\[
 -\frac{n}{(p_j)_{\infty}}<\beta_\infty^j+\sum_{k=1}^l \beta_k^j<\frac{n}{(p_j)_{\infty}^\prime}.
\]
So $w_j\in V_{p_j(\cdot)}(\bbR^n, \Pi)$.
By Lemma~\ref{last lemma}, $\mathcal{M}$ is bounded on $L^{(\tilde{p_j})^\prime(\cdot)}(w_j^{-q_j(\tilde{p_j})^\prime (\cdot)})$.
Again, by the assumption, we get
\begin{equation}\label{last inquality1}
 -\sum_{j=1}^N\frac{n}{p_j(x_k)}<\sum_{j=1}^N \beta_k^j<\frac{n}{p^\prime(x_k)}, \quad k=1,\cdots,l,
 \end{equation}
 \begin{equation}\label{last inquality2}
 -\sum_{j=1}^N\frac{n}{(p_j)_{\infty}}<\sum_{j=1}^N \beta_\infty^j+\sum_{k=1}^l\sum_{j=1}^N \beta_k^j< \frac{n}{p_\infty^\prime}.
\end{equation}
Note that the left side of (\ref{last inquality1}) and (\ref{last inquality2}) are equal to $-\frac{n}{p(x_k)}$ and $-\frac{n}{p_{\infty}}$, respectively. So $w_1\cdots w_N \in V_{p(\cdot)}(\bbR^n, \Pi)$.

 By of \cite[Proposition~2.3]{cruz}, we know $\frac{1}{p_j(\cdot)}\in LH_0$. Therefore, $\frac{1}{p(\cdot)}=\frac{1}{p_1(\cdot)}+\cdots+\frac{1}{p_N(\cdot)} \in LH_0$. Thus $p(\cdot)\in LH_0$.
Now by Lemma~\ref{last lemma}, $\mathcal{M}$ is bounded on $L^{(\tilde{p})^\prime(\cdot)}({(w_1\cdots w_N)}^{-q(\tilde{p})^\prime (\cdot)})$. By Theorem~\ref{hormander}, there exist some $C>0$ such that
 \[
 \|\mathrm{T}_m(\vec{f})\|_{L^{p(\cdot)}(w_1^{p(\cdot)}\cdots w_N^{p(\cdot)})} \leq C \prod_{j=1}^N\|f_j\|_{L^{p_j(\cdot)}(w_j^{p_j(\cdot)})}.
 \]
\end{proof}

 \end{document}